\documentclass[11pt,oneside,reqno]{amsart}
\numberwithin{equation}{section}
\newtheorem{thm}{Theorem}[section]
\newtheorem{lem}[thm]{Lemma}
\newtheorem{prp}[thm]{Proposition}
\newtheorem{cor}[thm]{Corollary}
\theoremstyle{definition}

\theoremstyle{remark}
\newtheorem*{ex*}{Example}
\newtheorem*{nt*}{Notation}
\newtheorem*{rm*}{Remark}
\usepackage{enumerate}
\usepackage{amssymb}
\usepackage{geometry}
\usepackage[colorlinks=true,allcolors=blue]{hyperref}
\usepackage{tikz-cd}
\newcommand{\al}{\alpha}
\newcommand{\be}{\beta}
\newcommand{\ga}{\gamma}
\newcommand{\de}{\delta}
\newcommand{\De}{\Delta}
\newcommand{\ep}{\varepsilon}
\newcommand{\la}{\lambda}
\newcommand{\BZ}{\mathbb{Z}}
\newcommand{\BC}{\mathbb{C}}
\newcommand{\Fgl}{\mathfrak{gl}}
\newcommand{\Fsl}{\mathfrak{sl}}
\newcommand{\Fsp}{\mathfrak{sp}}
\newcommand{\Fso}{\mathfrak{so}}
\newcommand{\Fosp}{\mathfrak{osp}}

\newcommand{\Fb}{\mathfrak{b}}
\newcommand{\Fg}{\mathfrak{g}}

\newcommand{\Fh}{\mathfrak{h}}
\newcommand{\Fk}{\mathfrak{k}}

\newcommand{\Fp}{\mathfrak{p}}

\newcommand{\Fr}{\mathfrak{r}}
\newcommand{\CA}{\mathcal{A}}
\newcommand{\CB}{\mathcal{B}}
\newcommand{\CD}{\mathcal{D}}
\newcommand{\CZ}{\mathcal{Z}}
\DeclareMathAlphabet{\mathbbit}{U}{bbm}{m}{sl}
\newcommand{\II}{\mathbbit{I}} 
\newcommand{\defby}{\stackrel{\text{\tiny def}}{=}}
\newcommand{\dm}[1]{\underset{#1}{\diamond}}

\renewcommand{\tilde}{\widetilde}
\renewcommand{\bar}{\overline}

\DeclareMathOperator{\DR}{DR}

\DeclareMathOperator{\Id}{Id}

\title[Parabolic Stabilization and Cutting for Reduction Superalgebras]{Parabolic Stabilization and Cutting\\for Reduction Superalgebras}
\author{Jonas T. Hartwig}
\date{}
\address{Department of Mathematics, Iowa State University, Ames, IA-50011, USA}
\email{jth@iastate.edu}
\urladdr{http://jthartwig.net}
\thanks{This work was supported by the Army Research Office grant W911NF-24-1-0058.}

\begin{document}

\begin{abstract}
The diagonal reduction algebra of a reductive Lie algebra $\mathfrak{g}$ is a localization of the Mickelsson algebra associated to the symmetric pair $(\mathfrak{g}\times\mathfrak{g},\, \mathfrak{g})$. In 2010, Khoroshkin and Ogievetsky introduced the methods of \emph{stabilization and cutting}, which relate the commutation relations in the diagonal reduction algebra of $\mathfrak{gl}_m\oplus\mathfrak{gl}_n$ with those in the diagonal reduction algebra of $\mathfrak{gl}_{m+n}$. We extend this method to a wide range of reduction algebras, including all diagonal and differential reduction algebras for basic classical Lie superalgebras. We show how the method can be used for computing relations in the diagonal reduction algebra of $\mathfrak{so}_8$ and differential reduction algebra of $\mathfrak{sp}_{2n}$.
\end{abstract}

\maketitle

\section{Introduction}

If $\Fg$ is a semisimple Lie subalgebra of a finite-dimensional complex Lie algebra $\tilde\Fg$, and $V$ is a finite-dimensional representation of $\tilde\Fg$, then $V$ is completely reducible as a $\Fg$-module, by Weyl's theorem. The decomposition of $V$ into a direct sum of irreducible $\Fg$-modules can be directly written down once we find an $\Fh$-weight basis for the space of primitive vectors $V^+=\{v\in V\mid \Fg_+v=0\}$, where $\Fg=\Fg_-\oplus\Fh\oplus\Fg_+$ is some fixed triangular decomposition. The vector space $V^+$ carries extra structure coming from the universal enveloping algebra $A=U(\tilde\Fg)$. Namely, let $I_+=A\Fg_+$ be the left ideal in $A$ generated by $\Fg_+$. Let $N=N_A(I_+)=\{a\in A\mid I_+a\subset I_+\}$ be the normalizer of $I_+$ in $A$. Then $N$ is a subalgebra of $A$ containing $I_+$ as a two-sided ideal. It is an exercise to check that $NV^+\subset V^+$ and $I_+V^+=0$. Therefore $V^+$ is a left module over $N/I_+$.

The algebra $S=S(\tilde\Fg,\Fg)=N/I_+$ is Mickelsson's \emph{step algebra} introduced in \cite{M73}. Because $S$ is difficult to describe, he defined a subalgebra $S^\circ$ of $S$ and conjectured that when $V$ is an irreducible representation of $\tilde\Fg$, the space $V^+$ is irreducible as a left $S^\circ$-module. This was proved in \cite{V75}.

In a different approach, also in the early 1970's, the extremal projector method was developed (see \cite{T11} and references therein). This provides a universal map $P:V\to V^+$ expressed in the Chevalley-Serre generators of $\Fg$. Such projectors exist also for basic classical Lie superalgebras, and for quantum groups.

Zhelobenko pointed out that $N/I_+=(A/I_+)^+$, and therefore elements of Mickelsson's step algebra can be accessed using the extremal projector applied to $A/I_+$ viewed as a representation of $\Fg$. This perspective was explored in many papers throughout the 1980's and 1990's, see e.g. \cite{Z89} for the main points.

At this stage we see that we can forget that $A$ was an enveloping algebra. The only thing we need is that there is a homomorphism $\varphi:U(\Fg)\to A$ where $\Fg$ is semisimple (or more generally, reductive), as long as certain technical properties are satisfied that were true before. The resulting algebras have been called generalized Mickelsson algebras, denoted $S(A,\Fg)$, but a more recent term is \emph{reduction algebras}. For the properties $A$ should satisfy, see for example \cite[§3.1]{KO08}.

As the focus in the literature has shifted, or at least broadened, away from the representation theory of $\Fg$ itself to the structure of these reduction algebras for their own sake, a number of important advances has taken place in recent years.

The first is that it is easier to describe the reduction algebra if the image of integer-shifted coroots $h_\be+m$ ($\be\in\De_+, m\in\BZ$) are invertible in $A$. This has to do with the expression for the extremal projector. It can easily be achieved by Ore localization, denoted by $'$. Such localization commutes with the normalizer construction: $S'(A,\Fg)=S(A',\Fg)$. Therefore the original reduction algebra $S(A,\Fg)$ is just a subalgebra.

A second innovation is the double coset realization. Let $I_-=\Fg_-A$ be the right ideal in $A$ generated by $\Fg_-$. Already in \cite[Thm.~1]{M73}, Mickelsson proved that the map (in his setting)
\begin{equation}\label{eq:map}
S(A,\Fg)=N/I_+ \to I_-\backslash A/I_+ = A/(I_-+I_+)
\end{equation}
given by inclusion $N/I_+\to A/I_+$ followed by canonical projection, is injective. This is a remarkable fact, as we are throwing away a lot of information when modding out by $I_-$. In \cite[Thm.~1]{K04}, Khoroshkin proved that, in the localized setting, \eqref{eq:map} is an isomorphism. It is much easier to write down elements of the double coset space $A/\II$, $\II=I_-+I_+$, than elements of $N/I_+$: any element of $A$ can serve as a coset representative. On the other hand, as proved in \cite[Thm.~1]{K04}, the (unique) multiplication $\diamond$ on the double coset space that makes the linear isomorphism \eqref{eq:map} an algebra isomorphism is given by
\begin{equation}
(a+\II) \diamond (b+\II) = (aPb)+\II.
\end{equation}
Here $P$ is the extremal projector for $\Fg$, expressed as an infinite series. This is the trade-off: In $N/I_+$, elements are hard but multiplication is easy. In $A/\II$, elements are easy but multiplication is hard.

In the same paper, \cite{K04}, it was shown that the extremal projector is equivalent to the universal dynamical twist $J$ from \cite{ABRR98} which satisfies the so-called ABRR equation. Besides providing a link to dynamical quantum groups, this gave a new recursive formula for the extremal projector, and hence for the multiplication in reduction algebras.

Another powerful tool is an action of the braid group for $\Fg$ by algebra automorphisms on the reduction algebra \cite{KO08}. This was developed in a general setting for any reductive Lie algebra $\Fg$. In particular, it applies to both diagonal reduction algebras $\DR(\Fg)$ (associated to $\Fg\subset\Fg\times\Fg$) and differential reduction algebras $\CD(\Fg)$ (associated to $U(\Fg)\to W\otimes U(\Fg)$ where $W$ is some Weyl algebra).

Despite all these methods and tools, to the best of our knowledge, the only Lie (super)algebras $\Fg$ where a complete presentation of the diagonal reduction algebra $\DR(\Fg)$ by generators and relations is known are $\Fg=\Fgl_n$ \cite{OK10,KO11,KO17} and $\Fg=\Fosp_{1|2}$ \cite{HW22}. 
Although it was discovered in \cite{KO17} that the relatively involved presentation of $\DR(\Fgl_n)$ from \cite{OK10,KO11} can expressed in dynamical $R$-matrix formalism as single relation, the reflection equation, it is unclear to what extent this can be done in other cases.

Lastly, in \cite{OK10,KO11}, a new method was invented called \emph{Stabilization and Cutting}. Together with the braid group action from \cite{KO08}, these were the main tools in their derivation of the presentation of $\DR(\Fgl_n)$.
The method was formulated only in the setting relevant for their proof, which concerned how commutation relations in $\DR(\Fgl_{n-1})$ compare with those in $\DR(\Fgl_n)$. For many reasons it is convenient to keep the Cartan the same, so really what they looked at was $\DR(\Fgl_m\oplus\Fgl_n)$ and $\DR(\Fgl_{m+n})$, with the case $m=1$ of special (inductive) interest.
To remove any possible source confusion, we emphasize that there are now \emph{four} Lie algebras involved, forming a commutative square
\begin{equation}
\begin{tikzcd}
\Fgl_{m+n}\arrow[r]& \Fgl_{m+n}\times \Fgl_{m+n}\\
\Fgl_m\oplus\Fgl_n\arrow[u]\arrow[r]& (\Fgl_m\oplus\Fgl_n)\times(\Fgl_m\oplus\Fgl_n)\arrow[u]
\end{tikzcd}
\quad\leadsto\quad 
\begin{tikzcd}
 \DR(\Fgl_{m+n})\\
 \DR(\Fgl_m\oplus\Fgl_n)\arrow[u]
\end{tikzcd}
\end{equation}
This does lead to a \emph{linear} map $i:\DR(\Fgl_m\oplus\Fgl_n)\to\DR(\Fgl_{m+n})$ of reduction algebras, as indicated in the diagram. However, $i$ is not an algebra map in general. Stabilization and Cutting can be viewed as a way to measure the extent to which $i$ fails to be an algebra map. More concretely, Stabilization says that $i(xy)-i(x)i(y)$ always belongs to a certain intersection $J$ of a left and a right ideal. Cutting goes the other way, it says that some relations in the big reduction algebra can be ``cut'' (certain terms removed) to provide a commutation relation in the smaller algebra.

It is natural to expect, as suggested at the end of \cite[§3]{OK10}, that this method could be extended to other such commutative squares of Lie algebras. The main result of the present paper is the formulation of a natural and general setting in which the method of stabilization and cutting works well.

We wanted to cover not only diagonal reduction algebras, but also allow other examples such as differential reduction algebras. The second goal was to make sure as much as possible works in the super case. One reason for these goals is the appearance of such reduction algebras in connection to field equations in physics \cite{HUW25,DHW25}. The diagram that summarizes our setting and generalizes the above square is the following:
\begin{equation}
\begin{tikzcd}
U(\Fg)\arrow[r]& B\rtimes U(\Fg)\\
U(\Fk)\arrow[u]\arrow[r]& B_0\rtimes U(\Fk)\arrow[u]
\end{tikzcd}
\quad\leadsto\quad 
\begin{tikzcd}
 \CZ_\Fg\\
 \CZ_\Fk\arrow[u]
\end{tikzcd}
\end{equation}
In this diagram, $\Fk$ is the Levi subalgebra of a parabolic subalgebra of a finite-dimensional complex basic classical Lie superalgebra $\Fg$, $B$ is a left $U(\Fg)$-module algebra, and $B_0$ is a $U(\Fk)$-submodule subalgebra of $B$. The algebra $\CZ_\Fg$ is the double coset realization of the the localized reduction algebra associated to the homomorphism $U(\Fg)\to B\rtimes U(\Fg)$, $u\mapsto 1\otimes u$. Similarly for $\CZ_\Fk$ associated to $U(\Fk)\to B_0\rtimes U(\Fk)$. The appearance of module algebras is explained in the remark of Section \ref{sec:mod}. We refer to Section \ref{sec:setup} for full details.

The main result of the paper is Theorem \ref{thm:stabcut} which is a formulation of Stabilization and Cutting in this setting. As one application, we prove that Cutting provides an algebra map from the (ghost) center of $\CZ_\Fg$ to the (ghost) center of $\CZ_\Fk$, generalizing \cite[§4]{KO11}. The ghost center for the diagonal reduction algebra of $\Fosp_{1|2}$ was computed in \cite{HW23}. The center of differential reduction algebras has been described in \cite{HO17,H18}. For the diagonal reduction algebra of $\Fgl_n$ a conjectured generating set was given in \cite{KO17}.

We also consider some illustrative examples. The case of $\Fso_8$ is particularly interesting, as removing any node from the Dynkin diagram $D_4$ gives a union of type $A$ diagrams. Although we do not complete the presentation in this paper, we show in Section \ref{sec:so8} that the most complicated relations (relations of weight $0$) in $\DR(\Fso_8)$ can be computed using a combination of Stabilization and Cutting, braid group actions, and the known presentation of diagonal reduction algebras of type $A$ from \cite{OK10,KO11}. A future presentation of $\DR(\Fso_8)$ can then serve as the basis for further induction arguments to other simply laced types.
Finally, the differential reduction algebra $\CD(\Fsp_4)$ was computed in \cite{HW24}. We show in Section \ref{sec:sp2n} how to use this and Stabilization to find relations for $\Fsp_{2n}$.

\section{Setup}
\label{sec:setup}

We work over the field $\BC$ of complex numbers. All vector spaces, algebras, and representations are assumed to be complex and $\otimes=\otimes_\BC$. When we say ``space'', ``subspace'', ``algebra'', etc., we mean ``superspace'', ``subsuperspace'', ``superalgebra''.

\subsection{Lie Algebra Data: \texorpdfstring{$(\Fg,\Fh,\De_+, S)$}{(g,h,Delta,S}}
\label{sec:lie}
Let $\Fg$ be a finite-dimensional basic\footnote{there is a non-degenerate even supersymmetric invariant bilinear form on $\Fg$} classical\footnote{$\Fg_{\bar 1}$ is completely reducible as a $\Fg_{\bar 0}$-module} Lie superalgebra.
Fix a Cartan subalgebra $\Fh\subset \Fg_{\bar 0}$.
Let $\Fg=\Fh\oplus(\bigoplus_{\al\in\De} \Fg_\al)$ be the root space decomposition, where $\De=\De(\Fg)\subset\Fh^\ast$ is the set of roots.
Fix a subset of positive roots $\De_+\subset\De$, by definition satisfying $\De=\De_+\sqcup(-\De_+)$ and $(\De_++\De_+)\cap\De\subset\De_+$. Put $\De_-=-\De_+$. 
Let $\Fg_\pm=\bigoplus_{\al\in\De_\pm}\Fg_\al$ be the corresponding positive (negative) nilpotent part of $\Fg$, and $\Fb_\pm=\Fh\oplus\Fg_\pm$ the positive (negative) Borel subalgebra of $\Fg$.
Let $\Pi=\De_+\setminus(\De_++\De_+)$, the set of simple roots.
Lastly, select a subset $S\subset\Pi$. Then we get set decompositions
\[\De=\De^S_-\sqcup \De^S_0\sqcup \De^S_+,\qquad
\De^S_0\defby \De\cap \BZ S,\qquad \De^S_\pm \defby \De_\pm\setminus\De^S_0,\]
and therefore a vector space decomposition
\[\Fg=\Fr_-\oplus\Fk\oplus\Fr_+,\qquad \Fk=\Fh\oplus \bigoplus_{\al\in \De^S_0}\Fg_\al,\qquad \Fr_\pm=\bigoplus_{\al\in\De^S_\pm}\Fg_\al.\]
We emphasize that $\Fk$ and $\Fg$ have the same Cartan subalgebra $\Fh$. We put \[\Fp_\pm = \Fk\oplus\Fr_\pm.\] 
The subspaces $\Fp_\pm$ are the positive (respectively negative) parabolic subalgebra of $\Fg$ associated to $S$, $\Fk$ is the Levi subalgebra, and $\Fr_\pm$ are the nilradicals of $\Fp_\pm$.

\begin{ex*}
\begin{enumerate}[1.]
\item 
The following is the case considered in \cite{OK10}. Let $m$ and $n$ be positive integers.
Let $\Fg=\Fgl_{m+n}$, and $\Fh$ be the set of diagonal matrices in $\Fg$, $\De_+=\{\ep_i-\ep_j\}_{i<j}$, $\ep_i(e_{jj})=\de_{ij}$. Then $\Pi=\{\ep_i-\ep_{i+1}\}_{i=1}^{m+n-1}$. Take $S=\Pi\setminus\{\ep_m-\ep_{m+1}\}$. Then
\[
\Fk=\Fgl_m\oplus\Fgl_n,\qquad
\Fr_+=
\begin{bmatrix}
0_{m\times m} & \ast_{m\times n}   \\
0_{n\times m} & 0_{n\times n}
\end{bmatrix}, \qquad 
\Fr_-=
\begin{bmatrix}
0_{m\times m} & 0_{m\times n}   \\
\ast_{n\times m} & 0_{n\times n}
\end{bmatrix}.
\]
\item Examples of $\Fk\subset\Fg$ obtained by deleting an end vertex from the Dynkin diagram include
\[\Fosp_{1|2n-2}\oplus\Fgl_{0|1}\subset\Fosp_{1|2n},\quad\Fgl_n=\Fgl_{0|n}\subset\Fosp_{1|2n},\quad\Fgl_1\oplus\Fsp_{2n-2}\subset\Fsp_{2n},\quad\Fgl_n\subset\Fsp_{2n}.\]
\item $\Fsp_{2n}=\Fsp_{0|2n}$ is \emph{not} the Levi of any parabolic in $\Fosp_{1|2n}$ because they have equal rank.
\end{enumerate}
\end{ex*}

\begin{nt*}
Let $\theta:\Fg\to\Fg$ be the Cartan anti-automorphism. It swaps $\Fg_\al$ and $\Fg_{-\al}$ for $\al\in\De$, is the identity on $\Fh$, satisfies $\theta([x,y])=(-1)^{|x||y|}[\theta(y),\theta(x)]$ for homogeneous $x,y\in \Fg$ where $x\in\Fg_{|x|}$, and $\theta^2=\Id_{\Fg_{\bar 0}}\oplus (-\Id_{\Fg_{\bar 1}})$ (the parity automorphism of $\Fg$).
Let $U(\Fg)$ be the universal enveloping algebra of $\Fg$. For any algebra $A$ containing $U(\Fh)$, such that $D=U(\Fh)\setminus\{0\}$ is an Ore denominator set in $A$, we let $A'=D^{-1}A$ denote the localization. Let $\bar U(\Fg)$ be the Taylor extension of $U'(\Fg)$. Let $P_\Fg\in \bar U(\Fg)$ be the extremal projector for $\Fg$. It is the unique element of $\bar U(\Fg)$ satisfying
\begin{equation}\label{eq:P}
[\Fh,P_\Fg]=0, \qquad \theta(P_\Fg)=P_\Fg,\qquad \Fg_+P_\Fg=0,\qquad P_\Fg\equiv 1 \mod \Fg_-\bar U(\Fg)\cap \bar U(\Fg)\Fg_+.
\end{equation}
From these properties it is easy to deduce that $P_\Fg\Fg_-=0$ and $(P_\Fg)^2 = P_\Fg$. For details we refer to \cite{T11} and references therein.
Let $Q_+=\BZ_{\ge 0}\Pi$, the positive cone in the root lattice for $\Fg$.
Choose an ordered basis of root vectors for $\Fg$. Without loss of generality $\Fr_-<\Fk<\Fr_+$.
For $\la\in Q_+$, let $\CB_\la$ denote the corresponding basis for the weight space $U(\Fg_+)[\la]$. By \eqref{eq:P}, $P_\Fg$ may be expanded as follows:
\begin{equation}
P_\Fg = \sum_{\la\in Q_+} P_\Fg^\la,\qquad P_\Fg^\la = \sum_{e,e'\in\CB_\la} \theta(e')\xi_{e,e'}e
\end{equation}
where $\xi_{e,e'}\in U'(\Fh)$.
\end{nt*}

We need the following relation between extremal projectors for $\Fg$ and $\Fk$. We have not seen it in the literature in this generality.

\begin{prp}\label{prp:Pk}
$P_\Fg \equiv P_\Fk \mod \Fr_- \bar U(\Fg) \cap \bar U(\Fg)\Fr_+$.
\end{prp}

\begin{proof}
Consider the truncation
\begin{equation}
\tilde P_\Fk = \sum_{\la\in Q_+(\Fk)} P_\Fg^\la.
\end{equation}
It suffices to show that $\tilde P_\Fk=P_\Fk$. We do this by verifying that $\tilde P_\Fk$ satisfies the properties \eqref{eq:P} of an extremal projector for $\Fk$. The only property that is not immediate is that $\Fk_+ \tilde P_\Fk=0$.
Let $J=\Fr_- \bar U(\Fg)\cap \bar U(\Fg)\Fr_+$.
By our choice of ordered basis for $\Fg$,
\begin{equation}
P_\Fg-\tilde P_\Fk \in J.
\end{equation}
Note that $J$ is a $U(\Fk)$-bimodule, since $\Fk\Fr_-=\Fr_-\Fk+[\Fk,\Fr_-]\subset \Fr_-\Fk+\Fr_-$ and likewise on the right. In particular, $\Fk_+(P_\Fg-\tilde P_\Fk)\subset J$. On the other hand, $\Fk_+(P_\Fg-\tilde P_\Fk)\subset \bar U(\Fk)$ since $\Fk_+P_\Fg\subset \Fg_+P_\Fg=0$.
Consequently,
\begin{equation}
\Fk_+\tilde P_\Fk = \Fk_+(P_\Fg-\tilde P_\Fk) \subset \bar U(\Fk)\cap J.
\end{equation}
By the PBW theorem for $U(\Fg)$, the intersection $\bar U(\Fk)\cap J$ is zero.
\end{proof}

\subsection{Module Algebra Data: \texorpdfstring{$(B,B_0,B_\pm)$}{(B,B0,B+-)}}
\label{sec:mod}

We keep all notation from Section \ref{sec:lie}.
Let $B$ be a $U(\Fg)$-module algebra.
Equivalently, $B$ is an associative algebra on which $\Fg$ acts by derivations.
We denote the action of $u\in U(\Fg)$ on $b\in B$ by $u\cdot b$.
We assume that $B$ is a weight module with respect to $\Fh$ and that $\Fg$ acts locally finitely on $B$.
We further assume that $B$ has a decomposition
\begin{equation}\label{eq:B}
B=B_0\oplus (B_- + B_+)
\end{equation}
such that
\begin{enumerate}[{\rm 1.}]
\item $B_0$ is a subalgebra and a $\Fk$-submodule of $B$,
\item $B_+$ is a finitely generated left ideal, a right $B_0$-submodule, and a $\Fp_+$-submodule of $B$,
\item $B_-$ is a finitely generated right ideal, a left $B_0$-submodule, and a $\Fp_-$-submodule of $B$,
\item $\Fr_+\cdot B_0\subset B_+$ and $\Fr_-\cdot B_0\subset B_-$.
\end{enumerate}
We call this a \emph{triangular decomposition of $B$ compatible with $(\De_+,S)$}.
Choose finite-dimensional generating subspaces $M_\pm\subset B_\pm$. Since the action of $\Fg$ on $B$ is locally finite, we may assume without loss of generality that $M_\pm$ is an $\Fh$-weight $\Fp_\pm$-submodule of $B_\pm$.

\begin{ex*}
\begin{enumerate}[{\rm 1.}]
\item For any $(\Fg,\Fh,\De_+,S)$, $B=U(\Fg)$ has a compatible triangular decomposition
\begin{equation}\label{eq:U-tri}
U(\Fg)=U(\Fk)\oplus(\Fr_-U(\Fg)+ U(\Fg)\Fr_+).
\end{equation}
\item The $n$:th Weyl algebra $W(2n)$ is a $U(\Fsp_{2n})$-module algebra. Choosing $\De_+=\{\ep_i-\ep_j\}_{i<j}\cup\{\ep_i+\ep_j\}_{i,j=1}^n$ gives $\Pi=\{\ep_1-\ep_2,\,\cdots,\,\ep_{n-1}-\ep_n,\,2\ep_n\}$. Let $S=\Pi\setminus\{\ep_1-\ep_2\}$.
Then $\Fk\cong\Fgl_1\oplus\Fsp_{2n-2}$ and a compatible triangular decomposition is 
\begin{equation}\label{eq:D-tri}
W(2n)=W(2n-2)\oplus (\partial_1 W(2n)+ W(2n)x_1)
\end{equation}
where $W(2n-2)$ denotes the subalgebra of $W(2n)$ generated by $\{x_i,\,\partial_i\}_{i=2}^n$.
\end{enumerate}
\end{ex*}

\subsection{Pairs of Reduction Algebras} 

Given $(\Fg,\Fh,\De_+,S; B,B_0,B_\pm)$, let $A_\Fg$ be the smash product $B\rtimes U(\Fg)$.
By definition, (i) $A_\Fg$ contains $B$ and $U(\Fg)$ as subalgebras, (ii) the multiplication map $B\otimes U(\Fg)\to A_\Fg$, $a\otimes b\mapsto ab$ is a vector space isomorphism, and (iii) the cross relation $xb = (-1)^{|x||b|} bx + x\cdot b$ holds all for homogeneous $x\in\Fg, b\in B$, where $\cdot$ is the $\Fg$-module action on $B$.
These properties determine $A_\Fg$ uniquely (up to isomorphism) as an associative algebra.
Define the double coset space
\begin{equation}
\CZ_\Fg = A_\Fg'/\II_\Fg,\qquad \II_\Fg = \Fg_-A_\Fg' + A_\Fg'\Fg_+,
\end{equation}
equipped with the diamond product
\begin{equation}\label{eq:diamond-product}
(a+\II_\Fg)\dm{\Fg}(b+\II_\Fg) = \sum_{\la\in Q_+} \big(a P_\Fg^\la b + \II_\Fg),\qquad\forall a,b\in A_\Fg'.
\end{equation}
Since $B$ is locally $\Fg$-finite, only finitely many terms in \eqref{eq:diamond-product} are nonzero.
The algebra $(\CZ_\Fg,\dm{\Fg})$ is an associative algebra.
It is isomorphic to $N/I$ where $I=A_\Fg'\Fg_+$ and $N$ is the normalizer of $I$ in $A_\Fg'$.
Likewise, let $A_\Fk=B_0\rtimes U(\Fk)$ and
\begin{equation}
\CZ_\Fk = A_\Fk'/\II_\Fk,\qquad \II_\Fk = \Fk_-A_\Fk'+A_\Fk'\Fk_+,
\end{equation}
equipped with the diamond product
\begin{equation}
(a+\II_\Fk)\dm{\Fk}(b+\II_\Fk) = \sum_{\la\in Q_+(\Fk)} (aP_\Fk^\la b+\II_\Fk),\qquad \forall a,b\in A_\Fk'.
\end{equation}
Here $Q_+(\Fk)=\BZ_{\ge 0}S$. It is the positive cone in the root lattice for $\Fk$.
The inclusion $A_\Fk'\to A_\Fg'$ takes $\II_\Fk$ into $\II_\Fg$ and therefore induces an injective linear map 
\begin{equation}\label{eq:inclusion}
i:\CZ_\Fk\to \CZ_\Fg.
\end{equation}
In general, this is not an algebra map.

\begin{ex*}
\begin{enumerate}[1.]
\item Following Example 1 from Section \ref{sec:mod}, we get a linear inclusion of diagonal reduction algebras $i:\DR(\Fk)\to \DR(\Fg)$.
\item Following Example 2 from Section \ref{sec:mod}, we get a linear inclusion of differential reduction algebras $i:\CD(\Fgl_1\oplus\Fsp_{2n-2})\to \CD(\Fsp_{2n})$.
\end{enumerate}
\end{ex*}

\begin{rm*}
If $B$ is an associative algebra and $\varphi:U(\Fg)\to B$ is an algebra homomorphism, then $B$ becomes a left $\Fg$-module by defining $x\cdot b = [\varphi(x),b]$ for $x\in\Fg$ and $b\in B$.
One may now localize $B$ and build a double coset algebra like above. But, we get a better behaved\footnote{More precisely, the reduction algebra will be a flat deformation of $U'(\Fh)\otimes B$. And in any case the counit provides a map $U(\Fg)\otimes B\to B$ which induces a surjective algebra map on reduction algebras.} algebra if we instead consider $\CA=U(\Fg)\otimes B$ and use $\Phi:U(\Fg)\to\CA$, $\Phi=(\Id\otimes\varphi)\circ\De$, where $\De$ here denotes the comultiplication $\De:U(\Fg)\to U(\Fg)\otimes U(\Fg)$.
Then $\Phi$ can be extended to an isomorphism of associative algebras
\begin{equation}
A=B\rtimes U(\Fg)\longrightarrow \CA
\end{equation}
by requiring $b\mapsto 1\otimes b$ for $b\in B$. Under this map, the denominator set $D\subset U(\Fh)$ is mapped to $\Phi(D)$. Therefore we get an induced surjection $A'\to\CA'$ where $\CA'$ is the localization of $\CA$ at the denominator set $\Phi(D)$. Furthermore, the left and right ideals $A'\Fg_+$, $\Fg_- A'$, are mapped to $\CA'\Phi(\Fg_+)$, $\Phi(\Fg_-)\CA'$ and likewise the extremal projector is mapped to the one where all root vectors $e_\be$ have been replaced by $\Phi(e_\be)$. The conclusion is that the ``traditional'' reduction algebra $\CA'/(\Phi(\Fg_-)\CA'+\CA'\Phi(\Fg_+))$ is isomorphic to $A'/\II_\Fg$. Therefore, working with module algebras is both more general while providing a simpler description of objects like $\II_\Fg$, the denominator set, the extremal projector in the diamond product, etc. (as no homomorphism $\varphi$ is needed; the action of $U(\Fg)$ on $B$ is built-in to the algebra structure of the smash product $B\rtimes U(\Fg)$).
\end{rm*}

\section{Stabilization and Cutting}
We keep all notation from Section \ref{sec:setup}.
Let $\overline{M_\pm}=\pi(M_\pm)$ where $\pi:A'_\Fg\to\CZ_\Fg$ is the canonical projection.
Let $V_+$ be the left ideal of $\CZ_\Fg$ generated by $\overline{M_+}$.
Let $V_-$ be the right ideal of $\CZ_\Fg$ generated by $\overline{M_-}$.
The following generalizes \cite[Lemma~5]{OK10},\cite[Lemma~4]{KO11}.

\begin{lem}\label{lem:V-and-V'}
The following equalities hold in the reduction algebra $\CZ_\Fg$:
\begin{enumerate}[{\rm (i)}] 
\item $V_+=\pi(A'_\Fg M_+)$,
\item $V_-=\pi(M_-A'_\Fg)$.
\end{enumerate}
\end{lem}

\begin{proof}
For part (i), we first show that $V_\Fk\subset \pi(A'_\Fg M_+)$. 
Since $A'_\Fg$ is unital, we have $\overline{M_+}\subset \pi(A'_\Fg M_+)$.
So, by definition of $V_+$, it suffices to show that $\pi(A'_\Fg M_+)$ is a left ideal of $\CZ_\Fg$.
For any $a,b\in A'_\Fg$ and $x\in M_+$, we have
\[\bar a\dm{\Fg} \pi(bx) = \sum_{\la\in Q_+(\Fg)} \overline{aP_\Fg^\la bx}\]
with at most finitely many nonzero terms.
Each of those terms have the form $\overline{cx}$ for some $c\in A'_\Fg$, and $\overline{cx}=\pi(cx)\in\pi(A'_\Fg M_+)$.
This proves the $\subset$ inclusion.

For the reverse inclusion, recall that $M_+$ is a finite-dimensional weight $\Fp_+$-submodule of $B_+$.
We will work with the decomposition of $M_+$ into weight spaces:
\[M_+=\bigoplus_{\be\in\Fh^\ast}M_+[\be],\qquad M_+[\be]=\{x\in M_+\mid h\cdot x=\la(h)x,\,\forall h\in\Fh\}.\]
We will use the partial order $\le$ on $\Fh^\ast$ given by $\la\le \mu$ iff $\mu-\la\in Q_+(\Fg)$.
We use $\la<\mu$ to mean $\la\le\mu$ and $\la\neq\mu$.
This restricts to a partial order on the set $\De(M_+)=\{\la\in\Fh^\ast\mid M_+[\la]\neq 0\}$ of weights.
We prove by induction that all weights $\be$ of $M_+$ satisfy
\begin{equation}\label{eq:lem-pf-inclusion}
\pi\big(A'_\Fg (M_+[\be])\big)\subset 
\sum_{\mu\in Q_+(\Fg)}\CZ_\Fg\dm{\Fg}\overline{M_+[\be+\mu]}.
\end{equation}
For the base case we consider the set of maximal elements of $\De(M_+)$.
These weights $\be$ necessarily satisfy $\Fg_+\cdot M_+[\be]=0$.
Let $z\in M_+[\be]$ and $a\in A'_\Fg$.
Then $\bar a\dm{\Fg} \overline{z}=\overline{aP_\Fg z}=\overline{az}$ because $e_{\ga}\cdot z=0$ for all positive roots $\ga$.
Thus \eqref{eq:lem-pf-inclusion} holds in this case.

For the induction step, suppose \eqref{eq:lem-pf-inclusion} has been proved for all weights in some subset $\Omega\subset\De(M_+)$. Let $\be$ such that $\be+(Q_+\setminus\{0\})\cap\De(M_+)\subset\Omega$.
For any $z\in M_+[\be]$ and $a\in A'_\Fg$,
\begin{align*}
\bar a\dm{\Fg} \bar z &= az+\sum_{0\neq\la\in Q_+} a P_\Fg^\la z + \II_\Fg\\
&=az+\sum_{0\neq\la\in Q_+}\sum_{e,e'\in\CB_\la} a\theta(e')\xi_{e,e'}e z +\II_\Fg\\ 
&=az+\sum_{0\neq\la\in Q_+}\sum_{e,e'\in\CB_\la} a\theta(e')\xi_{e,e'} (e\cdot z) + \II_\Fg,
\end{align*}
The weight of $e\cdot z$ equals $\be+\la$. Therefore, in nonzero terms, $\be+\la\in\Omega$.
Therefore, $\overline a\dm{\Fg}\overline z - \pi(az)$ belongs to the right hand side of \eqref{eq:lem-pf-inclusion} by the induction hypothesis applied to each nonzero term.
Consequently $\pi(az)$ also belongs to the right hand side of \eqref{eq:lem-pf-inclusion}. 
This proves the induction step.
Since \eqref{eq:lem-pf-inclusion} holds for all weights $\be$ of $M_+$, and the right hand side of \eqref{eq:lem-pf-inclusion} is contained in $V_+$, we conclude that $\pi(A'_\Fg M_+)\subset V_+$.
Part (ii) is proved in the same way.
\end{proof}

\begin{prp}\label{prp:hc}
Let $I = V_- + V_+$. Then
\begin{equation}\label{eq:I}
\CZ_\Fg = I \oplus i(\CZ_\Fk).
\end{equation}
\end{prp}

\begin{proof}
The inclusion $U'(\Fh)\otimes B\to A'_\Fg$ followed by canonical projection $A'_\Fg\to\CZ_\Fg$ gives an isomorphism of left $U'(\Fh)$-modules
\[U'(\Fh)\otimes B \overset{\simeq}{\longrightarrow} \CZ_\Fg.\]
Under this map, the image of $U'(\Fh)\otimes B_0$ equals $i(\CZ_\Fk)$.
By Lemma \ref{lem:V-and-V'}, the image of $U'(\Fh)\otimes B_\pm$ equals $V_\pm$.
\end{proof}

The following generalizes the methods of Stabilization and Cutting introduced in \cite{OK10,KO11} for the special case $(\Fg,\Fk)=(\Fgl_{m+n},\,\Fgl_m\oplus\Fgl_n)$ and $B=U(\Fgl_{m+n})$:

\begin{thm}\label{thm:stabcut}
Let $I=V_-+V_+$ and $J=V_-\cap V_+$.
\begin{enumerate}[{\rm (i)}]
\item (Stabilization) Suppose that we have a relation in $\CZ_\Fk$ of the form
\begin{subequations}
\begin{equation}\label{eq:stabilizing-assumption}
\sum_j x_j\dm{\Fk}y_j=0.
\end{equation}
where $x_j, y_j\in\CZ_\Fk$. Then there is an element $z\in J$ such that the  relation
\begin{equation}\label{eq:stabilizing-conclusion}
\sum_j i(x_j)\dm{\Fg}i(y_j)=z,
\end{equation}
\end{subequations}
holds in $\CZ_\Fg$, where $i:\CZ_\Fk\to\CZ_\Fg$ is the linear inclusion from \eqref{eq:inclusion}.
\item (Cutting) Suppose that we have a relation in $\CZ_\Fg$ of the form
\begin{subequations}
\begin{equation}\label{eq:cutting-assumption}
\sum_j i(x_j)\dm{\Fg}i(y_j)=u,
\end{equation}
where $x_j,y_j\in \CZ_\Fk$ and $u\in I$. Then, in $\CZ_\Fk$,
\begin{equation}\label{eq:cutting-conclusion}
\sum_j x_j\dm{\Fk}y_j=0
\end{equation}
\end{subequations}
and furthermore, it is necessarily the case that $u\in J$.
\end{enumerate}
\end{thm}

\begin{proof}
(i) It suffices to show that for all $x,y\in A'_\Fk$, and all $\la\in Q_+$,
$x P_\Fg^\la y + \II_\Fg$ is congruent to $x P_\Fk^\la y+\II_\Fg$ modulo $V_+$ and modulo $V_-$.
Write $\la=\la'+\la''$ where $\la'\in Q_+(\Fk)=\BZ_{\ge 0}S$, and $\la''\in\BZ_{\ge 0}(\Pi\setminus S)$.
If $\la''=0$, then $x P_\Fg^\la y +\II_\Fg=x P_\Fk^\la y+\II_\Fg$, by Proposition \ref{prp:Pk}.
If $\la''\neq 0$, by choice of ordered basis for $\Fg$ (in which $\Fr_-<\Fk<\Fr_+$ see Section \ref{sec:lie}), each term in $xP_\Fg^\la y$ has the form $ae_\ga y$ for some $a\in A'_\Fg$ and some root vector $e_\ga\in\Fr_+$.
Modulo $\II_\Fg$ we can replace multiplication by action:
$ae_\ga y+\II_\Fg=a(e_\ga \cdot y)+\II_\Fg$. By assumption on $B$, we have $\Fr_+\cdot B_0\subset B_+$.
Therefore each term with $\la''\neq 0$ belongs to $\pi(A_\Fg'B_+)$ which also equals $\pi(A_\Fg'M_+)$.
By Lemma \ref{lem:V-and-V'}, the latter equals $V_+$.
The proof that $\sum_{\la''\neq 0} xP_{\Fg}^\la y+\II_\Fg \in V_-$ is analogous.

(ii) Suppose that in $\CZ_\Fg$ we have \eqref{eq:cutting-assumption}. Let $v=\sum_j x_j\dm{\Fk}y_j$.
This can be written as the following relation in $\CZ_\Fk$:
\begin{equation}\label{eq:pf-cutting0}
\sum_j x_j\dm{\Fk}y_j + (-1)\dm{\Fk}v = 0.
\end{equation}
Thus, by part (i), we have
\begin{equation}\label{eq:pf-cutting1}
\sum_j i(x_j)\dm{\Fg}i(y_j) - 1\dm{\Fg}i(v) = z \in J_\Fk.
\end{equation}
But, by assumption in \eqref{eq:cutting-assumption}, we also have
\begin{equation}\label{eq:pf-cutting2}
\sum_j i(x_j)\dm{\Fg}i(y_j)=u\in I_\Fk.
\end{equation}
Combining \eqref{eq:pf-cutting1} and \eqref{eq:pf-cutting2} we obtain
\begin{equation}\label{eq:pf-cutting3}
z + i(v) = u.
\end{equation}
Since $z\in J$ and $u\in I$ we have $i(v)\in I$.
By Proposition \ref{prp:hc}, we conclude that $i(v)=0$.
Since $i$ is injective, $v=0$.
Substituting $v=0$ into \eqref{eq:pf-cutting0}, we obtain \eqref{eq:cutting-conclusion}.
Likewise, substituting $v=0$ into \eqref{eq:pf-cutting3}, we conclude that $u=z\in J$.
\end{proof}

\begin{cor} \label{cor:bimodule}
The following bimodule-like\footnote{The ``-like'' is referring to the fact that $i(\CZ_\Fk)$ is not a subalgebra of $\CZ_\Fg$.} properties of $V_\pm$ hold:
\begin{enumerate}[{\rm (i)}]
\item $V_+\dm{\Fg} i(\CZ_\Fk)\subset V_+$ and $i(\CZ_\Fk)\dm{\Fg}V_-\subset V_-$,
\item $i(\CZ_\Fk)\dm{\Fg}I\dm{\Fg}i(\CZ_\Fk)\subset I$.
\end{enumerate}
\end{cor}

\begin{proof}
Let $a\in A'_\Fg, b\in B_+, c\in B_0$.
Let $\la\in Q_+(\Fg)$.
As in the proof of stabilization,
\begin{equation}
abP_\Fg^\la c+\II_\Fg = ab P_\Fk^\la c + v + \II_\Fg
\end{equation}
for some $v\in A'_\Fg B_+$.
This comes from $\Fr_+\cdot B_0\subset B_+$.
Furthermore, since $B_+B_0\subset B_+$,
\begin{equation}
bP_\Fk^\la c \in B_+ A_\Fk' \subset A_\Fk' B_+\subset A_\Fg' B_+.
\end{equation}
Therefore we conclude that
\begin{equation}
abP_\Fg^\la c+\II_\Fg \in \pi(A_\Fg' B_+)=V_+.
\end{equation}
Hence $V_+\dm{\Fg}i(\CZ_\Fk)\subset V_+$.
The other case is proved analogously.

(ii): This is immediate by part (i), since $I=V_-+V_-$ and $V_+$ (respectively $V_-$) is a left (respectively right) ideal in $\CZ_\Fg$.
\end{proof}

\section{Cutting the Center}

The center and anti-center of an algebra $A$ are given by
\[Z(A)_i=\{z\in A_i\mid za-(-1)^{|z||a|}az=0\text{ for homogeneous $a\in A$}\},\] 
\[AZ(A)_i=\{z\in A_i\mid za-(-1)^{(\bar 1-|z|)|a|}az=0\text{ for homogeneous $a\in A$}\}.\]
The anti-center is a bimodule over the center, and $AZ(A)AZ(A)\subset Z(A)$. The \emph{ghost center} of $A$, introduced in \cite{G00}, is the algebra
\[GZ(A)=Z(A)\oplus AZ(A).\]
Recall the direct sum decomposition \ref{eq:I}.
Consider the linear map 
\begin{equation}
p:\CZ_\Fg\to \CZ_\Fk
\end{equation}
defined as the projection along $I$ followed by inverse of $i$.
Inspired by \cite[Proposition~7]{KO11}, we prove the following corollary.

\begin{cor}
The image of the (ghost) center of $\CZ_\Fg$ under $p$ is contained in the (ghost) center of $\CZ_\Fk$.
Furthermore, $p|_{GZ(\CZ_\Fg)}$ is an algebra homomorphism.
\end{cor}

\begin{proof}
Even though $V_\pm$ are not two-sided ideals of $\CZ_\Fg$,
\begin{equation}\label{eq:pf-c}
GZ(\CZ_\Fg)\dm{\Fg}V_\pm=V_\pm\dm{\Fg}GZ(\CZ_\Fg)\subset V_\pm.
\end{equation}
Indeed, for homogeneous $z\in GZ(\CZ_\Fg)$ and $v\in V_+$ we have $v\dm{\Fg}z=\pm z\dm{\Fg}v\in V_+$ and similarly $z\dm{\Fg}V_-\subset V_-$.
Thus, for $z\in GZ(\CZ_\Fg)$ and $a\in\CZ_\Fk$,
\[ip(z)\dm{\Fg}i(a)\pm i(a)\dm{\Fg}ip(z)\equiv z\dm{\Fg}i(a)\pm i(a)\dm{\Fg}z\equiv 0\mod I\] for appropriate sign.
Thus, by cutting, $p(z)\dm{\Fk}a\pm a\dm{\Fk}p(z)=0$.
This shows $p(z)\in GZ(\CZ_\Fk)$.
If $z,w\in GZ(\CZ_\Fg)$, then \eqref{eq:pf-c} implies that $ip(z)\dm{\Fg}ip(w)\equiv z\dm{\Fg}w\mod I$.
By stabilization, $ip(z)\dm{\Fg}ip(w)\equiv i(p(z)\dm{\Fk}p(w))\mod J$. Thus $p(z\dm{\Fg}w)=pi(p(z)\dm{\Fk}p(w))=p(z)\dm{\Fk}p(w)$.
\end{proof}

\section{Examples}

\subsection{Some Extreme Cases}

\begin{enumerate}[1.]
\item If $S=\Pi$, then $\Fk=\Fg$ and $\Fr_\pm=0$.
If $B$ is a $U(\Fg)$-module which decomposes as $B=B_0\oplus B_1$ where $B_0$ is a $\Fg$-stable subalgebra, $B_1$ is a $\Fg$-stable $B_0$-subbimodule, then we may take $B_+=B_1$, $B_-=0$.
In this case $J=V_+\cap V_- =0$, and stabilization expresses the fact that in this case $i:\CZ_\Fk\to \CZ_\Fg$ actually is an algebra map.
(Recall that $\CZ_\Fk$ is defined in terms of $B_0$.)
An example of this situation is for $\Fg=\Fgl_n$, $B=W(2n)$, $B_0=W(2n)_0$ and $B_1=\oplus_{d\neq 0} W(2n)_d$, where $W(2n)_d$ denotes the $d$-eigenspace of the (adjoint action of the) Euler operator $\sum_i x_i\partial_i$.
\item For any $(\Fg,\Fh,\De_+,S)$ and $U(\Fg)$-module algebra $B$, a compatible triangular decomposition is to take $B_0=B$, $B_\pm=0$.
Then $\CZ_\Fg$ may be regarded as a deformation of $\CZ_\Fk$ controlled by the bimodule $J$.
For example, put $\Fg_k = \Fgl_k\oplus (\Fgl_1)^{\oplus (n-k)}\subset\Fgl_n$ and consider the chain $U(\Fg_1)\subset U(\Fg_2)\subset\cdots\subset U(\Fg_n)\to W(2n)$.
To this chain corresponds a sequence of reduction algebras $\CZ_{\Fg_1}\to \CZ_{\Fg_2}\to\ldots\to\CZ_{\Fg_n}$, each of which is a flat deformation of $U'(\Fh)\cong\BC(t_1,\ldots,t_n)$ tensor the $n$:th Weyl algebra $W(2n)$, and each map is a linear isomorphism.
The algebra $\CZ_{\Fg_1}$ is undeformed (isomorphic to $U'(\Fh)\otimes W(2n)$) and $\CZ_{\Fg_n}$ is the most deformed version --- ``the'' differential reduction algebra $\CD(\Fgl_n)$ from \cite{HO17}.
\end{enumerate}

\subsection{Relations in \texorpdfstring{$\DR(\Fso_8)$}{DR(so(8))}}
\label{sec:so8}
Consider the simple Lie algebra $\Fso_8$. Let $\Pi=\{\al_0,\al_1,\al_2,\al_3\}$ where $\al_0$ is the central node in the Dynkin diagram $D_4$.
We have $4$ maximal parabolics given by $S^{(i)}=\Pi\setminus\{\al_i\}$ for $i=0,1,2,3$.
The corresponding Levis $\Fk^{(i)}$ are
\begin{equation}
\Fk^{(0)}\cong (\Fsl_2)^{\oplus 3}\oplus\Fgl_1,\qquad \Fk^{(i)}\cong \Fgl_4,\quad i=1,2,3.
\end{equation}
Each of these is a Lie algebra of ``$\Fgl$ type'', and therefore complete presentations of their diagonal reduction algebras are known \cite{OK10,KO11,KO17}.
We illustrate how one can combine stabilization, cutting, and the braid group action \cite{KO08}, to calculate relations in $\DR(\Fso_8)$.

It is known that $\DR(\Fg_1\times\Fg_2)\cong D^{-1}(\DR(\Fg_1)\otimes\DR(\Fg_2))$, where $D=U(\Fh_1\times\Fh_2)\setminus\{0\}$.
Therefore $\DR(\Fk^{(0)})$ is a localization of $\DR(\Fsl_2)^{\otimes 3}\otimes \DR(\Fgl_1)$.\footnote{Since there are no nilpotent parts, $P_{\Fgl_1}=1$ and $\DR(\Fgl_1)=D^{-1}U(\Fgl_1\times \Fgl_1)\cong \BC(h)[T]$.}
In particular,
\begin{equation}
E_i\dm{\Fk^{(i)}} F_i = H_i + \xi_2(h_i)F_i\dm{\Fk^{(i)}}E_i + \xi_0(h_i)H_i\dm{\Fk^{(i)}} H_i,
\end{equation}
where $E_i = \overline{e_i\otimes 1}, F_i=\overline{f_i\otimes 1}, H_i=\overline{h_i\otimes 1}$, and $i=1,2,3$.
The coefficients $\xi_0$ and $\xi_2$ are known (and can be calculated by hand, using the extremal projector).
By symmetry it suffices to consider the $i=1$ case.
Stabilizing this relation to $\DR(\Fso_8)$ gives, by Theorem \ref{thm:stabcut}(i),
\begin{equation}\label{eq:so8-1}
E_1\dm{\Fg} F_1 = H_1 + \xi_2(h_1)F_1\dm{\Fg}E_1 + H_1 + \xi_0(h_1)H_1\dm{\Fg} H_1 + \sum_{\be\in\De(\Fr_+)} \xi_\be F_\be \dm{\Fg} E_\be
\end{equation}
where $E_\be = \overline{e_\be\otimes 1}$, and similarly for $F_\be$.
(The bar now stands for coset modulo $\II_\Fg$ rather then $\II_{\Fk^{(1)}}$, also in the definition of $E_1,F_1,H_1$.)
The roots $\De(\Fr_+)=\De_+^{S^{(0)}}$ of the nilradical $\Fr_+$ are all the positive roots of $\Fso_8$ with positive coefficient of $\al_0$:
\begin{equation}
\De_+^{S^{(0)}} = \{\al_0\}\cup \{\al_0+\al_i\}_i \cup \{\al_0+\al_i+\al_j\}_{i\neq j}\cup \{\al_0+\al_1+\al_2+\al_3, \, 2\al_0+\al_1+\al_2+\al_3\}.
\end{equation}
By symmetry, only $5$ of the coefficients need to be determined:
\begin{equation}\label{eq:so8-c}
\xi_\be\quad\text{ for }\be\in\{\al_0,\,\al_0+\al_1,\, \al_0+\al_1+\al_2,\,\al_0+\al_1+\al_2+\al_3,\, 2\al_0+\al_1+\al_2+\al_3\}.
\end{equation}
This is what we learn from stabilizing the weight zero relations using $S^{(0)}=\Pi\setminus\{\al_0\}$.
Next, we will cut using $S^{(23)}=\Pi\setminus\{\al_2,\al_3\}=\{\al_0,\al_1\}$.
The corresponding Levi is $\Fk^{S^{(23)}}\cong\Fgl_3\oplus\Fgl_1$.
The roots of the nilradical here are all positive roots involving $\al_2$ and $\al_3$.
So, cutting \eqref{eq:so8-1} deletes those terms from the $\be$ sum, leaving only
\begin{equation}\label{eq:so8-2}
E_1\dm{\Fk} F_1 = H_1 + \xi_2(h_1)F_1\dm{\Fk}E_1 + H_1 + \xi_0(h_1)H_1\dm{\Fk} H_1 + \xi_{\al_0} F_{\al_0} \dm{\Fk} E_{\al_0} + \xi_{\al_0+\al_1} F_{\al_0+\al_1} \dm{\Fk} E_{\al_0+\al_1},
\end{equation}
where $\Fk=\Fk^{S^{(23)}}$. The critical point here is that \eqref{eq:so8-2} is a relation in $\DR(\Fgl_3\oplus\Fgl_1)$ and therefore two ($\xi_{\al_0}$ and $\xi_{\al_0+\al_1}$) out of the five coefficients from \eqref{eq:so8-c} are explicitly known from \cite[Eq. (6.27)]{KO11}.
The remaining three coefficients can be obtained by applying braid group automorphisms from \cite{KO08}:
First, apply $\check q_{\al_2}$ to both sides of \eqref{eq:so8-1}.
Since $s_{\al_2}(\al_0)=\al_0+\al_2$ and $s_{\al_2}(\al_0+\al_1)=\al_0+\al_1+\al_2$, this determines the two coefficients $\xi_{\al_0+\al_2}$ and $\xi_{\al_0+\al_1+\al_2}$.
Next apply $\check q_{\al_3}$ to get $\xi_{\al_0+\al_1+\al_2+\al_3}$, and lastly apply $\check q_{\al_0}$ to obtain the highest root coefficient $\xi_{2\al_0+\al_1+\al_2+\al_3}$.

\subsection{Differential Reduction Algebra of \texorpdfstring{$\Fsp_{2n}$}{sp(2n)}}
\label{sec:sp2n}
Presentations by generators and relations of the differential reduction algebra associated to $U(\Fgl_n)\to W(2n)$, $e_{ij}\mapsto x_i\partial_j$ have been computed (and more generally $U(\Fgl_n)\to U(\Fgl_{nN})\to W(2nN)$, see \cite{H18}).
In \cite{HW24}, a presentation of the reduction algebra associated to the homomorphism $\varphi:U(\Fsp_4) \to W(4)$ was given.
The general case of $U(\Fsp_{2n})\to W(2n)$ has not been computed.
Let $\CD(\Fsp_{2n})$ be the differential reduction algebra.
Explicitly, $\CD(\Fsp_{2n})=D^{-1}(W(2n)\rtimes U(\Fsp_{2n}))/\II_{\Fsp_{2n}}$, $D=U(\Fh)\setminus\{0\}$, equipped with the diamond product.
Here we comment on the application of stabilization to this problem, using the embedding $\Fk=\Fgl_1\oplus\Fsp_{2n-2}\subset\Fsp_{2n}=\Fg$. 
A compatible triangular decomposition of $B=W(2n)$ is
\begin{equation}
W(2n)=W(2n-2)\oplus (\partial_1 W(2n)+ W(2n)x_1).
\end{equation}
Take now $n=3$.
One of the relations in the reduction algebra $\CD(\Fsp_4)$ is
\[\bar x_1\dm{\Fsp_4}\bar\partial_1 = \xi_{10} + \xi_{11}\bar\partial_1\dm{\Fsp_4}\bar x_1 + \xi_{12}\bar\partial_2\dm{\Fsp_4}\bar x_2\]
where $\bar x_i = x_i\otimes 1+\II_{\Fsp_4}$, $\bar \partial_i = \partial_i\otimes 1+\II_{\Fsp_4}$, and certain coefficients $\xi_{ij}\in U'(\Fh_{\Fsp_4})=\BC(h_1,h_2)$, computed explicitly in \cite{HW24}.
Let us lift this relation to $\CD(\Fsp_6)$ using stabilization.
With the choices as in Example 2 in Section \ref{sec:mod}, we should use index set $\{2,3\}$ for $W(4)$. So in $\CD(\Fgl_1\oplus\Fsp_4)$ we have
\[\bar x_2\dm{\Fk}\bar\partial_2 = \xi_{20} + \xi_{22}\bar\partial_2\dm{\Fk}\bar x_2 + \xi_{23}\bar\partial_3\dm{\Fk}\bar x_3\]
with coefficients in $\BC(h_2,h_3)$.
By Theorem \ref{thm:stabcut}(i), in $\CD(\Fsp_6)$ we have
\[\bar x_2\dm{\Fg}\bar\partial_2 = \xi_{20} + \Xi_{21} \bar\partial_1\dm{\Fg}\bar x_1 + \xi_{22}\bar\partial_2\dm{\Fg}\bar x_2 + \xi_{23}\bar\partial_3\dm{\Fg}\bar x_3.\]
A single new term has appeared, with unknown coefficient $\Xi_{21}\in U'(\Fh_{\Fsp_6})=\BC(h_1,h_2,h_3)$.
The other coefficients remain unchanged.
That is the meaning of ``stabilization''.
The coefficient $\Xi_{21}$ can be determined by using a braid group automorphism \cite{KO08} corresponding to the transposition $(1\,3)$.
In this way one can proceed inductively and determine a presentation for $\CD(\Fsp_{2n})$.

\end{document}